%% file: cobordisms.tex
\documentclass[12pt]{amsart}

\usepackage{amsmath, amsthm, amssymb, amsfonts, amscd}
\usepackage[latin1]{inputenc}
\usepackage[english]{babel}
\usepackage{hyperref, url}
\usepackage[all]{xy}
\usepackage{pinlabel}
\usepackage{color}

\usepackage[margin=3.6cm]{geometry}

\usepackage[nameinlink,capitalise,noabbrev]{cleveref} 

\newtheorem{thm}{Theorem}[section]
\newtheorem{lemma}[thm]{Lemma}
\newtheorem{prop}[thm]{Proposition}

\newtheorem{defn}[thm]{Definition}

\theoremstyle{remark}
\newtheorem{ex}[thm]{Example}
\newtheorem{rmk}[thm]{Remark}

\newcommand{\N}{\mathbb{N}}
\newcommand{\Z}{\mathbb{Z}}
\newcommand{\Q}{\mathbb{Q}}

\newcommand{\C}{\mathbb{C}}
\newcommand{\F}{\mathbb{F}}

\newcommand{\ft}{\mathfrak{t}}
\newcommand{\CFK}{{\rm CFK}^-}
\newcommand{\rCFK}{\underline{{\rm CFK}}^-}
\newcommand{\HFK}{{\rm HFK}^-}
\newcommand{\fs}{\mathfrak{s}}

\newcommand{\spinc}{spin$^c$~}

\newcommand{\del}{\partial}
\newcommand{\conc}{\mathcal{C}}
\newcommand{\Yt}{(Y,\ft)}

\newcommand{\HFp}{{\rm HF}^-}
\newcommand{\HFpred}{{\rm HF}^-_{\rm red}}

\title{A note on cobordisms of algebraic knots}
\author{J\'ozsef Bodn\'ar, Daniele Celoria, Marco Golla}
\date{}

\address{Alfr\'ed R\'enyi Institute of Mathematics, Budapest, Hungary}
\email{bodnar.jozef@renyi.mta.hu}

\address{Department di Mathematics, Universit\`a di Firenze, Firenze, Italy}
\email{celoria@mail.dm.unipi.it}

\address{Department of Mathematics, Uppsala University, Uppsala, Sweden}
\email{marco.golla@math.uu.se}

\begin{document}

\begin{abstract}
In this note we use Heegaard Floer homology to study smooth cobordisms of algebraic knots and complex deformations of cusp singularities of curves.
The main tool will be the concordance invariant $\nu^+$: 
we study its behaviour with respect to connected sums, providing an explicit formula in the case of L-space knots and proving subadditivity in general.
\end{abstract}

\maketitle

\input{sect/intro.tex}
\input{sect/prelim.tex}
\input{sect/nu.tex}
\input{sect/semicont.tex}
\input{sect/cob.tex}
\input{sect/subadd.tex}
\input{sect/examples.tex}

\bibliographystyle{amsplain}
\bibliography{cobordisms}

\end{document}

%% file: sect/intro.tex

\section{Introduction}

A \emph{cobordism} between two knots $K$, $K'$ in $S^3$ is a smoothly and properly embedded surface $F\subset S^3\times [0,1]$, with $\del F = K\times\{0\} \cup K'\times\{1\}$.
Carving along an arc connecting the two boundary components of $F$, one produces a slice surface for the connected sum $\overline K\#K'$, where $\overline K$ is the mirror of $K$.
Two knots are \emph{concordant} if there is a genus-0 cobordism between them; this is an equivalence relation, and the connected sum endows the quotient $\conc$ of the set of knots with a group operation; $\conc$ is therefore called \emph{concordance group}. A knot is \emph{smoothly slice} if it is concordant to the unknot.

Litherland \cite{Litherland} used Tristram--Levine signatures to show that torus knots are linearly independent in $\conc$. In fact, Tristram--Levine signatures provide a lower bound for the slice genus of knots.
Sharp lower bounds for the slice genus of torus knots are provided by the invariants $\tau$ in Heegaard Floer homology \cite{OzsvathSzabo-tau}, and $s$ in Khovanov homology \cite{Rasmussen-s}.

More recently, Ozsv\'ath, Stipsicz, and Szab\'o \cite{OzsvathSzaboStipsicz} defined the concordance invariant $\Upsilon$; Livingston and Van Cott \cite{LivingstonVanCott} used $\Upsilon$ to improve on the bounds given by signatures along some families of connected sums of torus knots.

In this note we consider \emph{algebraic knots}, i.e. links of irreducible curve singularities (\emph{cusps}), and more generally \emph{L-space knots}. Given two algebraic knots $K, L$, we give lower bounds on the genus of a cobordism between them by using the concordance invariant $\nu^+$ defined by Hom and Wu \cite{HomWu}. This is computed in terms of the semigroups of the two corresponding singularities, $\Gamma_K$ and $ \Gamma_L$, and the corresponding enumerating functions $\Gamma_K(\cdot)$ and $\Gamma_L(\cdot)$.

\begin{thm}\label{t:maincomputation}
If $K$ and $L$ are algebraic knots with enumerating functions $\Gamma_K(\cdot)$ and $\Gamma_L(\cdot)$ respectively, then:
\[
\nu^+(K\#\overline{L}) = \max\left\{g(K)-g(L) + \max_{n\ge 0}\{\Gamma_L(n) - \Gamma_K(n)\},0\right\}.
\]
\end{thm}

In \cref{ss:reducedHFK} we define an appropriate enumerating function for L-space knots;  \cref{t:maincompL} below mimics the statement above in this more general setting, and directly implies \cref{t:maincomputation}; the key of the definition and of the proofs is the reduced Floer complex defined by Krcatovich \cite{Krcat}.



As an application of \cref{t:maincomputation}, we give a different proof of a result of Gorsky and N\'emethi \cite{GorskyNemethi} on the semicontinuity of the semigroup of an algebraic knot under deformations of singularities, in the cuspidal case. A similar result was obtained by Borodzik and Livingston \cite{BorodzikLivingston-def} under stronger assumptions (see \cref{s:semicont} for details).

\begin{thm}\label{t:semicont}
Assume there exists a deformation of an irreducible plane curve singularity with semigroup $\Gamma_K$ to an irreducible plane curve singularity with semigroup $\Gamma_L$.
Then for each non-negative integer 
\[
\# (\Gamma_K \cap [0,n)) \leq \# (\Gamma_L \cap [0,n)).
\]
\end{thm}

Finally, we turn to proving some properties of the function $\nu^+$. The first one reflects analogous properties for other invariants (signatures, $\tau$, $s$, etc.)
 and gives lower bounds for the unknotting number and related concordance invariants (see \cref{s:cob} below).

\begin{thm}\label{p:crossingchanges}
If $K_+$ is obtained from $K_-$ by changing a negative crossing into a positive one, then
\[
\nu^+(K_-) \le \nu^+(K_+) \le \nu^+(K_-)+1.
\]
\end{thm}

\begin{thm}\label{t:subadd}
The function $\nu^+$ is subadditive. Namely, for any two knots $K, L\subset S^3$,
\[
\nu^+(K\#L) \le \nu^+(K) + \nu^+(L).
\]
\end{thm}

As an application, we consider some concordance invariants, also studied by Owens and Strle \cite{OwensStrle-unknotting}.
Recall that the \emph{concordance unknotting number} $u_c(K)$ of a knot $K$ is the minimum of unknotting numbers among all knots that are concordant to $K$; the \emph{slicing number} $u_s(K)$ of $K$ is the minimal number of crossing changes needed to turn $K$ into a slice knot; finally, the \emph{$4$--ball crossing number} $c^*(K)$ is the minimal number of double points of an immersed disc in the $4$--ball whose boundary is $K$. It is quite remarkable that there are knots for which these quantities disagree \cite{OwensStrle-unknotting}.

\begin{prop}\label{p:concordanceinvariants}
The unknotting number, concordance unknotting number, slicing number, and 4--ball crossing number of $K$ are all bounded from below by $\nu^+(K) + \nu^+(\overline K)$.
\end{prop}


\subsection{Organisation of the paper} In \cref{s:prelim} we recall some facts about Heegaard Floer correction terms and reduced knot Floer complex. In \cref{s:computenu} we prove \cref{t:maincomputation} as a corollary of \cref{t:maincompL}, and in \cref{s:semicont} we prove \cref{t:semicont}. In \cref{s:cob} we study cobordisms between arbitrary knots and prove \cref{p:crossingchanges} and \cref{p:concordanceinvariants}; in \cref{s:subadd} we prove \cref{t:subadd}. Finally, in \cref{s:examples} we study some concrete examples.

\subsection{Acknowledgments} We would like to thank Paolo Aceto, Maciej Borodzik, and Kouki Sato for interesting conversations, and Maciej Borodzik for providing us with some computational tools.
The first author has been supported by the ERC grant LDTBud at MTA Alfr\'ed R\'enyi Institute of Mathematics.
The third author was partially supported by the PRIN--MIUR research project 2010--11 ``Variet\`a reali e complesse: geometria, topologia e analisi armonica'' and by the FIRB research project ``Topologia e geometria di variet\`a in bassa dimensione''.

%% file: sect/prelim.tex

\section{Singularities and Heegaard Floer homology}\label{s:prelim}

\subsection{Links of curve singularities}

In what follows, $K$ and $L$ will be two \emph{algebraic knots}. We will recall briefly what this means and also what invariants can be associated with such knots. For further details, we refer to \cite{BrieskornKnorrer,EisenbudNeumann,Wall}.

Assume $F \in \mathbb{C}[x,y]$ is an \emph{irreducible} polynomial which defines an \emph{isolated irreducible plane curve singularity}.
This means that $F(0,0) = 0$ and in a sufficiently small neighbourhood $B_{\varepsilon} = \{ |x|^2+|y|^2 \leq \varepsilon^2 \}, \varepsilon > 0$ of the origin, $\partial_1 F(x,y) = \partial_2 F(x,y) = 0$ holds if and only if $(x,y) = (0,0)$.
The \emph{link} of the singularity is the zero set of $F$ intersected with a sphere of sufficiently small radius: $K = \{F(x,y) = 0\} \cap \partial B_{\varepsilon}$.
Since $F$ is irreducible, $K$ is a knot, rather than a link, in the 3--sphere $\partial B_{\varepsilon}$.
A knot is called \emph{algebraic} if its isotopy type arises in the above described way.
All algebraic knots are \emph{iterated torus knots}, i.e. they arise by iteratively cabling a torus knot.

The zero set of every isolated irreducible plane curve singularity admits a local parametrization, i.e. there exists $x(t), y(t) \in \mathbb{C}[[t]]$ such that $F(x(t),y(t)) \equiv 0$ and $t \mapsto (x(t),y(t))$ is a bijection for $|t| < \eta \ll 1$ to a neighbourhood of the origin in the zero set of $F$. Consider the following set of integers:
\[ \Gamma_K = \{\textrm{ord}_t\, G(x(t),y(t)) \mid F \nmid G \in \mathbb{C}[[x,y]] \}. \]
It can be seen easily that $\Gamma_K$ is an additive semigroup. It depends only on the local topological type of the singularity; therefore, it can be seen as an invariant of the isotopy type of the knot $K$. We will say that $\Gamma_K$ is the \emph{semigroup of the algebraic knot} $K$.

We denote with $\N = \{0,1,\dots\}$ the set of non-negative integers.
The semigroup $\Gamma_K$ is a cofinite set in $\N$; in fact, $|\N \setminus \Gamma_K| = \delta_K < \infty$ and the greatest element not in $\Gamma_K$ is $2\delta_K-1$. The number $\delta_K$ is called the $\delta$--\emph{invariant} of the singularity. It is well-known that $\delta_K$ is the Seifert genus of $K$: $\delta_K = g(K)$.

We also write $\Gamma_K(n)$ for the $n$--th element of $\Gamma_K$ with respect to the standard ordering of $\N$, with the convention that $\Gamma_K(0) = 0$. The function $\Gamma_K(\cdot)$ will be called the \emph{enumerating function} of $\Gamma_K$.

\subsection{Heegaard Floer and concordance invariants}

Heegaard Floer homology is a family of invariants of 3--manifolds introduced by Ozsv\'ath and Szab\'o \cite{OzsvathSzabo-HF}; in this paper we are concerned with the `minus' version over the field $\F=\Z/2\Z$ with two elements.
It associates to a rational homology sphere $Y$ equipped with a \spinc structure $\ft$ a $\Q$--graded $\F[U]$--module $\HFp\Yt$; the action of $U$ decreases the degree by 2.

The group $\HFp\Yt$ further splits as a direct sum of $\F[U]$--modules $\F[U]\oplus\HFpred\Yt$.
We call $\F[U]$ the \emph{tower} of $\HFp\Yt$.
The degree of the element $1\in \F[U]$ is called the \emph{correction term} of $\Yt$, and it is denoted by\footnote{
Note that our definition of $d\Yt$ would differ by 2 from the original definition of~\cite{OzsvathSzabo-absolutely}; however, it is more convenient for our purposes to use a shifted grading in $\HFp$.
}
$d\Yt$.

When $Y$ is obtained as 
an integral surgery along a knot $K$ in $S^3$, one can recover the correction terms of $Y$ in terms of a family of invariants introduced by Rasmussen \cite{Rasmussen-Goda} and then further studied by Ni and Wu \cite{NiWu}, and Hom and Wu \cite{HomWu}. We call these invariants $\{V_i(K)\}_{i\ge 0}$, adopting Ni and Wu's notation instead of Rasmussen's --- who used $h_i(K)$ instead --- as this seems to have become more standard.

Recall that there is an indexing of \spinc structures on $S^3_n(K)$ as defined in \cite[Section 2.4]{OzsvathSzabo-integralsurgeries}: $S^3_n(K)$ is the boundary of the surgery handlebody $W_n(K)$ obtained by attaching a single 2--handle with framing $n$ along $K\subset \del B^4$. 
Notice that, if we orient $K$ there is a well-defined generator $[F]$ of $H_2(W_n(K);\Z)$ obtained by capping off a Seifert surface of $K$ with the core of the 2--handle. 
The \spinc structure $\ft_k$ on $S^3_n(K)$ is defined as the restriction of a \spinc structure $\fs$ on $W_n(K)$ such that
\begin{equation}\label{e:spinclabel}
\langle c_1(\fs), [F] \rangle \equiv n+2k \pmod{2n}
\end{equation}

\begin{thm}[\cite{Rasmussen-Goda, NiWu}]
The sequence $\{V_i(K)\}_{i\ge 0}$ takes values in $\N$ and is eventually 0. Moreover, $V_{i}(K)-1 \le V_{i+1}(K) \le V_i(K)$ for every $i$.

With the spin$^c$ labelling defined in \eqref{e:spinclabel} above, for every integer $n$ we have
\begin{equation}\label{e:NiWu}
d(S^3_{n}(K),\ft_i) = -2\max\{V_{i}(K),V_{n-i}(K)\} + \frac{(n-2i)^2-n}{4n}
\end{equation}
\end{thm}

\begin{defn}[\cite{HomWu}]
The minimal index $i$ such that $V_i(K)=0$ is called $\nu^+(K)$.
\end{defn}

\subsection{Reduced knot Floer homology}\label{ss:reducedHFK}
In \cite{Krcat} Krcatovich introduced the \emph{reduced knot Floer complex} $\rCFK(K)$ associated to a knot $K$ in $S^3$.
This complex is graded by the \emph{Maslov grading} and filtered by the \emph{Alexander grading}; the differential decreases the Maslov grading by 1 and respects the Alexander filtration.

Without going into technical details, for which we refer to \cite{Krcat}, any knot Floer complex $\CFK(K)$ can be recursively simplified until the differential on the graded object associated to the Alexander filtration becomes trivial (while the differential on the filtered complex is, in general, nontrivial).
Moreover, $\rCFK(K)$ still retains an $\F[U]$--module structure.

The power of Krcatovich's approach relies in the application to connected sums; if we need to compute $\CFK (K_1 \# K_2) \cong \CFK (K_1) \otimes_{\F [U]} \CFK (K_2)$ we can first reduce $\CFK (K_1)$, and then take the tensor product ${\rm \underline{CFK}^-}(K_1)\otimes_{\F [U]} \CFK (K_2)$.

This is particularly effective when dealing with L-\emph{space knots}, i.e. knots that have a positive integral surgery $Y$ such that $\HFp\Yt = \F[U]$ for every \spinc structure $\ft$ on $Y$. Notice that all algebraic knots are L-space knots \cite[Theorem 1.8]{Hedden}.

In this case, $\rCFK(K)$ is isomorphic to $\F[U]$ as an $\F[U]$--module, and it has at most one generator in each Alexander degree. If we call $x$ the homogeneous generator of $\rCFK(K)$ as an $\F[U]$--module, then $\rCFK(K) = \F[U]x$, and $\{U^nx\}_{n\ge 0}$ is a homogeneous basis of $\rCFK(K)$.

We denote with $\Gamma_K(n)$ the quantity $g(K) - A(U^n\cdot x)$, where $A$ is the Alexander degree, and we call $\Gamma_K(\cdot)$ the \emph{enumerating function} of $K$.
As observed by Borodzik and Livingston \cite[Section 4]{BorodzikLivingston}, when $K$ is an algebraic knot, the function $\Gamma_K(\cdot)$ coincides with the enumerating function of the semigroup associated to $K$ as defined above.
Accordingly, we define the \emph{semigroup} of $K$ as the image of $\Gamma_K$.

\begin{ex}
Observe also that this is not the enumerating function of a semigroup in general; to this end, consider the pretzel knot $K = P(-2,3,7) = 12n_{242}$.
$K$ is an L-space knot with Alexander polynomial $t^{-5}-t^{-4}+ t^{-2}-t^{-1}+ 1-t+ t^2-t^4+ t^5$, hence the function $\Gamma_K(\cdot)$ takes values $0, 3, 5, 7, 8, 10, 11, 12, \dots$.
Since 3 is in the image of $\Gamma_K(\cdot)$ but 6 is not, $\Gamma_K(\cdot)$ is not the enumerating function of a semigroup.
\end{ex}

%% file: sect/nu.tex

\subsection{An example}

We are going to show an application of the reduced knot Floer complex in a concrete case. Consider the knot $K = T_{3,7}\#\overline{T_{4,5}}$. The genera, signatures, and $\upsilon$-function~\cite{OzsvathSzaboStipsicz} of $T_{3,7}$ and $T_{4,5}$ all agree: $g(T_{3,7}) = g(T_{4,5}) = 6$, $\sigma(T_{3,7}) = \sigma(T_{4,5}) = 8$, and $\upsilon(T_{3,7}) = \upsilon(T_{4,5}) = -4$. It follows that $\tau(K) = s(K) = \sigma(K) = \upsilon(K) = 0$. However, we can show the following.

\begin{prop}
The knot $K$ satisfies $\nu^+(K) = \nu^+(\overline K) = 1$.
\end{prop}

\begin{proof}
We need to compute a Floer complex of $T_{3,7}, T_{4,5}$ and their mirrors, as well as the reduced Floer complex of $T_{3,7}, T_{4,5}$. Call $K_1 = T_{3,7}$ and $K_2 = T_{4,5}$

For an L-space knot, and in particular for every positive torus knot, the knot Floer complex is determined by a \emph{staircase complex}, which in the case of $K_1$ and $K_2$ reads as follows:
\[
\begin{array}{ccc}\CFK(K_1) & & \CFK(K_2)\\
\xymatrixcolsep{0.5 pc}\xymatrixrowsep{0.5pc}\xymatrix{
\bullet & \bullet\ar[l]\ar[dd] & & & & &\\
& & & & & &\\
& \bullet & \bullet\ar[l]\ar[dd] & & & &\\
& & & & & &\\
& & \bullet & & \bullet\ar[ll]\ar[d] & &\\
& & & & \bullet & & \bullet\ar[ll]\ar[d]\\
& & & & & & \bullet
}
& \hphantom{abc} &
\xymatrixcolsep{0.5 pc}\xymatrixrowsep{0.5 pc}\xymatrix{
\bullet & \bullet\ar[l]\ar[ddd] & & & & &\\
& & & & & &\\
& & & & & &\\
& \bullet & & \bullet\ar[ll]\ar[dd] & & &\\
& & & & & &\\
& & & \bullet & & &\bullet\ar[lll]\ar[d] \\
& & & & & &\bullet \\
}
\end{array}
\]
The complexes for $\overline K_1$ and $\overline K_2$ are easily computed from these, and they, too, are staircases.

The reduced complex $\rCFK(K_1)$, on the other hand, has a single generator in each of the following bi-degrees $(-i,j)$ (where $-i$ records the $U$--power and $j$ records the Alexander grading):
\[
(0,6), (-1,3), (-2,0), (-3,-1), (-4,-3), (-5,-4), (-6-n,-6-n), n \ge 0.
\]
The reduced complex $\rCFK(K_2)$ has a generator in each of the following bi-degrees:
\[
(0,6), (-1,2), (-2,1), (-3,-2), (-4,-3), (-5,-4), (-6-n,-6-n), n \ge 0.
\]
In both cases, the $U$--action carries a generator with $i$--coordinate $k$ to one with $i$--coordinate $k-1$.
Taking the tensor product over $\F[U]$, one gets twisted staircases as follows, with a generator in bidegree $(0,0)$ (marked with a $\star$):
\[
\begin{array}{ccc}\rCFK(K_1)\otimes\CFK(\overline K_2) & & \rCFK(K_2)\otimes\CFK(\overline K_1)\\
\xymatrixcolsep{0.5 pc}\xymatrixrowsep{0.5 pc}\xymatrix{
& & &\circ\ar[ddlll]\ar[dd] & & \circ\ar[ddd]\ar[ddll] &\\
\bullet\ar[d] & & & & & & \star\ar[ddl] \\
\bullet & & & \bullet& & &\\
& & & & & \bullet &
} & \phantom{abc} &
\xymatrixcolsep{0.5 pc}\xymatrixrowsep{0.5 pc}\xymatrix{
& & & & \circ\ar[dd]\ar[ddll] & & \\
\bullet\ar[d] & & \bullet\ar[dll]\ar[d] & & & & \star\ar[dddl] \\
\bullet & & \bullet & & \bullet & \bullet\ar[l]\ar[dd] & \\
& & & & & & \\
& & & & & \bullet &
}
\end{array}
\]
The generators marked with a $\circ$ exhibit the fact that $V_0(K_1\#\overline{K}_2)$ and $V_0(K_2\#\overline{K}_1)$ are both strictly positive (see \cite[Section 4]{Krcat} for details).
\end{proof}

\section{Computing the invariant}\label{s:computenu}

In this section we are going to prove a version of \cref{t:maincomputation} for L-space knots. Given an integer $x$ we denote with $(x)_+$ the quantity $(x)_+ = \max\{0,x\}$.

\begin{thm}\label{t:maincompL}
Let $K$ and $L$ be two L-space knots with enumerating functions $\Gamma_K(\cdot),\Gamma_L(\cdot):\N\to\N$. Then
\[
\nu^+(K\#\overline{L}) = \left(g(K)-g(L) + \max_{n\ge 0}\{\Gamma_L(n) - \Gamma_K(n)\}\right)_+.
\]
\end{thm}

Notice that, since algebraic knots are $L$-space knots, \cref{t:maincomputation} is an immediate corollary. \cref{t:maincompL} will in turn be a consequence of the following proposition.

\begin{prop}\label{p:computenu}
In the notation of \cref{t:maincompL}, let $\{0=a_1 < \dots < a_d=g(L)\}$ be the image of the function $n\mapsto\Gamma_L(n)-n$, and define $a'_k = g(L)-a_{d+1-k}$ for $k=1,\dots,d$.
Then
\[
\nu^+(K\#\overline{L}) = \left(g(K)-g(L) + \max_{1\le k \le d}\{a_k + a'_k - \Gamma_K(a'_k)\}\right)_+.
\]
\end{prop}

\begin{proof}
Call $\delta_K = g(K)$, $\delta_L = g(L)$.
Consider the complex $\rCFK(K)\otimes_{\F[U]}\CFK(\overline L)$, that computes the knot Floer homology of $K\#\overline L$. 
Recall that the function $\Gamma_K(\cdot)$ describes the reduced Floer complex: $\rCFK(K)$ has a generator $x_k$ in each bidegree
$(-k,\delta_K-\Gamma_K(k))$.
Moreover, $U\cdot x_k = x_{k+1}$.

As observed by Krcatovich \cite[Section 4]{Krcat}, the sequences $\{a_k\}, \{a_k'\}$ determine a `twisted staircase' knot Floer complex $\CFK(\overline L)$ for $\overline L$: the generator of the tower $\F[U]$ in $\HFK(\overline L)$ is represented by the sum of $d$ generators $U^{a'_1}y_1,\dots,U^{a'_d}y_d$, where $y_k$ sits in bidegree $(0, a'_k+a_k-\delta_L)$.
In more graphical terms, $a_k$ will be the Alexander grading of $U^{a_k'}y_k$, i.e. its $j$--coordinate, and $-a_k'$ will be its $i$--coordinate.

The tensor product $\rCFK(K)\otimes\CFK(\overline L)$ has a staircase in Maslov grading 0 generated by the chain $z = \sum_{k=1}^d x_0\otimes U^{a_k'}y_k$.
Notice that $x_0\otimes U^{a_k'}y_k = x_{a_k'}\otimes y_k$ sits in Alexander degree $A(x_{a_k'}) + A(y_k) = \delta_K - \Gamma(a_k') + a_k+a_k' - \delta_L$.
Therefore, the maximal Alexander degree in the chain $z$ is precisely $M = \delta_K-\delta_L+\max\{a_k+a'_k - \Gamma_K(a_k)\}$, and we claim that if $M\ge 0$, then $\nu^+(K\#\overline L) = M$.

We let $A^-_k$ be the filtration sublevel of $C_\# = \rCFK(K)\otimes\CFK(\overline L)$ defined by $j\le k$, i.e. generated by all elements with Alexander filtration level at most $k$.

If $M \le k$, the entire staircase is contained in the subcomplex $A^-_k$. That is, the inclusion $A^-_k \to C_\#$ induces a surjection onto the tower, hence $\nu^+(K\#\overline L) \le k$.
In particular, if $M\le 0$, then $\nu^+(K\#\overline L) = 0 = (M)_+$.

If $M>0$, for each $k<M$ the complex $A^-_k$ misses at least one of the generators of the chain; this implies that the inclusion $A^-_k \to \CFK(K)$ does not induce a surjection onto the tower. 
It follows that $V_k(K\#\overline L) > 0$.
Hence, by definition of $\nu^+$, we have $\nu^+(K\#\overline L) = M = (M)_+$, as desired.
\end{proof}

\begin{proof}[Proof of \cref{t:maincompL}]
As remarked above, the values of $a_k$ and $a_k'$ determine the positions of the generators in the staircase. By the symmetry of the Alexander polynomial (and hence of the staircases), $\Gamma_L(a_k')-a_k' = a_k$ for each $k$ (compare with \cite[Section 4]{Krcat}).

Moreover, for any $a_k' \leq n < a_{k+1}'$, we have $\Gamma_L(n)-n = a_k$, and for every $a_d'\le n$ we have $\Gamma_L(n)-n = a_d$.
Furthermore, as $\Gamma_K(\cdot)$ is strictly increasing, $n \mapsto \Gamma_K(n)-n$ is non-decreasing, therefore for any $a_k' \leq n < a_{k+1}'$ we have $\Gamma_K(a_j')-a_j' \leq \Gamma_K(n)-n$, so
\begin{align*}
a_k + a_k' - \Gamma_K(a_k') &= (\Gamma_L(a_k') - a_k') - (\Gamma_K(a_k') - a_k') =\\
&= \max_{a_k' \leq n < a_{k+1}'} \{ (\Gamma_L(n) - n) - (\Gamma_K(n) - n) \} = \max_{a_k' \leq n < a_{k+1}'} \{ \Gamma_L(n) - \Gamma_K(n) \}.\qedhere
\end{align*}
\end{proof}

\begin{rmk}
The same argument shows that, for every $m\le V_0(K\#\overline L)$:
\[
\min\{i\mid V_i(K\#\overline L)=m\} = \left(g(K)-g(L) + \max_{n\ge0}\{\Gamma_L(n) - \Gamma_K(n+m)\}\right)_+,
\]
thus allowing one to compute all correction terms of $K\#\overline L$ from the enumerating functions of $K$ and $L$.
\end{rmk}

%% file: sect/semicont.tex

\section{Semicontinuity of the semigroups}\label{s:semicont}

In this section we prove \cref{t:semicont} about the deformations of plane curve singularities.
We note here that our \cref{t:semicont} differs slightly from both of the results mentioned in the introduction: it reproves \cite[Prop. 4.5.1]{GorskyNemethi} in the special case when both the central and the perturbed singularity are irreducible, but (in the spirit of \cite{BorodzikLivingston-def}) using only smooth topological (not analytic) methods; however, we do not restrict ourselves to $\delta$--constant deformations, as opposed to \cite[Theorem 2.16]{BorodzikLivingston-def}.

In the context of deformations, inequalities which hold for certain invariants are usually referred to as \emph{semicontinuity} of that particular invariant. Our result can be viewed as the semicontinuity of the semigroups (resembling the spectrum semicontinuity, cf. also \cite[Section 3.1.B]{BorodzikLivingston-def}).

For a brief introduction to the topic of deformations, we follow mainly \cite[Section 1.5]{BorodzikLivingston-def} and adapt the notions and definitions from there.
By a \emph{deformation} of a singularity with link $K$ we mean an analytic family $\{F_s\}$ of polynomials parametrised by $|s|<1$, such that there exists a ball $B\subset \C^2$ with the following properties:
\begin{itemize}
\item the only singular point of $F_0$ inside $B$ is at the origin;
\item $\{F_s=0\}$ intersects $\partial B$ transversely and $\{F_s=0\}\cap \partial B$ is isotopic to $K$ for every $|s|<1$;
\item the zero set of $F_s$ has only isolated singular points in $B$ for every $|s|<1$;
\item all the singular points of $F_s$ inside $B$ are irreducible for every $|s|<1$;
\item all fibres $F_s$ with $s\neq 0$ have the same collection of local analytic type of singularities.
\end{itemize}

For simplicity, we also assume that there is only one singular point of $F_s$ inside $B$ for each $s$. If such an analytic family of polynomials $\{F_s\}$ exists, we say that the singularity of $F_0$ at the origin \emph{has a deformation} to the singularity of $F_{1/2}$.

Consider now a sufficiently small ball $B_2$ around the singular point of $F_{s_0}$ for a fixed $0<|s_0|<1$
such that $\{F_{s_0} = 0\} \cap \partial B_2$ is isotopic to $L$, the link of the perturbed singular point. Then $V = \{F_{s_0} = 0\} \cap \overline{B} \setminus B_2$ is a genus-$g$ cobordism between $K$ and $L$, where $g = g(K) - g(L)$.

Let $K$, $L$ be two L-space knots, with corresponding semigroups $\Gamma_K$ and $\Gamma_L$, respectively. 
We define the \emph{semigroup counting functions} $R_{K}, R_{L}: \N\to\N$ as
$R_{K}(n) = \# [0,n) \cap \Gamma_K$ and $R_{L}(n) = \# [0,n) \cap \Gamma_L$.
For simplicity, we allow $n$ to run on negative numbers as well: if $n < 0$, then we define $R_{K}(n) = R_{L}(n) = 0$.  In this section, we will assume that $g(K) = \delta_K \ge \delta_L = g(L)$.

\begin{prop}\label{prop:semi}
Assume there is a genus-$g$ cobordism between two L-space knots $K$ and $L$. Then for any $a \in \mathbb{Z}$ we have
\[ R_{K}(a+\delta_K) \leq R_{L}(a+\delta_L+g). \]
\end{prop}
\begin{proof}
Since $\nu^{+}$ is a lower bound for the cobordism genus, by \cref{t:maincomputation} for any $m \in \N$ we have
\[ \delta_K - \delta_L + \Gamma_L(m) - \Gamma_K(m) \leq g, \]
equivalently,
\[ \Gamma_L(m) - \delta_L - g \leq \Gamma_K(m) - \delta_K. \]
Notice that since $\Gamma_K(m) = a$ implies $R_{K}(a) = m$, and the largest $a$ for which $R_{K}(a) = m$ is $a = \Gamma_K(m)$ (and analogously for $\Gamma_L$), the above inequality can be interpreted as
\[ R_{K}(a+\delta_K) \leq R_{L}(a+\delta_L+g).\qedhere \]
\end{proof}

The proposition above should be compared with \cite[Theorem 2.14]{BorodzikLivingston-def}. In \cite{BorodzikLivingston-def}, Borodzik and Livingston introduced the concept of \emph{positively self-intersecting concordance}, and \cite[Theorem 2.14]{BorodzikLivingston-def} is the counterpart of \cref{prop:semi} above: their assumption is on the double point count of the positively self-intersecting concordance, while ours is on the cobordism genus. The former is related to the 4--ball crossing number considered in \cref{p:concordanceinvariants}.

The assumption in \cite{BorodzikLivingston-def} allowed Borodzik and Livingston to treat $\delta$--constant deformations (because irreducible singularities can be perturbed to transverse intersections). However, equipped with \cref{prop:semi}, we can prove the semigroup semicontinuity even if the deformation is not $\delta$--constant (but assuming that there is only one singularity in the perturbed curve $\{F_{1/2} = 0\}$).

\begin{proof}[Proof of \cref{t:semicont}]
Apply \cref{prop:semi} with $a=n-\delta_K$ and recall that $g = \delta_K-\delta_L$ in this case. 
\end{proof}

\begin{rmk}
In \cite[Section 3]{BorodzikLivingston-def}, the example of torus knots $T_{6,7}$ and $T_{4,9}$ was extensively studied.
The semigroup semicontinuity proved in \cref{t:semicont} obstructs the existence of a deformation between the corresponding singularities.
Since the difference of the $\delta$--invariants is 3, a deformation from $T_{6,7}$ to $T_{4,9}$ would produce a genus-3 cobordism between the two knots. However, the bound coming from $\nu^{+}$ is $4$ (compare with \cite[Remark 3.1]{BorodzikLivingston-def}).
\end{rmk}

%% file: sect/cob.tex

\section{Bounds on the slice genus and concordance unknotting number}\label{s:cob}

Recall that $\nu^+(K) \le g_*(K)$ for every knot $K$;
as outlined in the introduction, this shows that $\nu^+(K\#\overline{L})$ gives a lower bound on genus of cobordisms between $K$ and $L$.
Notice that $\nu^+(L\#\overline K)$ gives a bound, too, and the two bounds are often different.

We now state a preliminary lemma that we will use to prove \cref{p:crossingchanges}, i.e. that trading a negative crossing for a positive one does not decrease $\nu^+$, nor does it increase it by more than 1.

\begin{lemma}\label{l:gvsnu}
If there is a genus-$g$ cobordism between two knots $K$ and $L$, then for each $m\ge0$ the following holds:
\[
V_{m+g}(K) \le V_m(L).
\]
As a consequence, $\nu^+(K) \le \nu^+(L)+g$.
\end{lemma}

This lemma can be compared with \cref{prop:semi}; indeed, using \cite[Equation (5.1)]{us}, the proposition can be restated as $V_m(K)\le V_{m+g}(L)+g$, for each $m\ge 0$.

\begin{proof}
Consider the 4--manifold $W$ obtained by attaching a 4--dimensional 2--handle to $S^3\times[0,1]$ along $L\times \{1\}\subset S^3\times \{1\}$, with framing $n\ge 2\nu^+(L)$.

The cobordism is a genus-$g$ embedded surface $F$ in $S^3\times[0,1]$, whose boundary components are $K\times\{0\}$ and $L\times\{1\}$.
Capping off the latter boundary component in $W$, and taking the cone over $(S^3\times\{0\},K)$ , we obtain obtain a singular genus-$g$ surface $\widehat F\subset W' = W \cup B^4$, whose only singularity is a cone over $K$.

As argued in \cite[Section 4]{us} and \cite[Theorem 3.1]{BorodzikHeddenLivingston}, the boundary $\partial N$ of a regular neighbourhood $N$ of $\widehat F$ in $W'$ is diffeomorphic to the 3--manifold $Y_n$ obtained as $n$--surgery along the connected sum of $K$ and the Borromean knot $K_{B,g}$ in $\#^{2g}(S^2\times S^1)$.
It follows that $Z = -(W'\setminus N$) can be looked at as a cobordism from $S^3_n(L)$ to $Y_n$.

\begin{figure}[h]
\includegraphics[scale=.7]{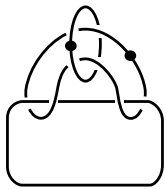}
\caption{The Borromean knot $K_{B,1}$. The Borromean knot $K_{B,g}$ is the connected sum of $g$ copies of $K_{B,1}$.}
\end{figure}

We can look at $N$ as the surgery cobordism from $\#^{2g}(S^2\times S^1)$ to $Y_n$, filled with a 1--handlebody;
since the class of $[\widehat F]$ generates both $H_2(N)$ and $H_2(-W')$, we obtain that the restriction of any \spinc structure on $-W'$ to $Z$ induces an isomorphism between (torsion) \spinc structure on its two boundary components that respects the surgery-induced labelling.
Moreover, we also obtain that $b_2^+(Z) = 0$.

The 3--manifold $Y_n$ has standard ${\rm HF}^\infty$ \cite{us, BorodzikHeddenLivingston}, and its bottom-most correction terms have been computed in \cite[Proposition 4.4]{us} and \cite[Theorem 6.10]{BorodzikHeddenLivingston}:
\[
d_b(Y_n,\ft_m) = \min_{0\le k\le g}\{2k-g-2V_{m+g-2k}(K)\} - \frac{n-(2m- n)^2}{4n}.
\]

We observe that choosing $k=0$ in the minimum we obtain the inequality:
\[
d_b(Y_n,\ft_m) \le -g-2V_{m+g}(K) - \frac{n-(2m-n)^2}{4n}.
\]

Applying the last inequality and \cite[Theorem 4.1]{Behrens} to $Z$, seen as a negative semidefinite cobordism from $S^3_n(L)$ to $Y_n$, we get:
\[
d(S^3_n(L),\ft_m) \le d_b(Y_n, \ft_m) + g,
\]

from which
\[
-2V_m(L) \le -g-2V_{m+g}(K)+g \Longleftrightarrow V_{m+g}(K) \le V_m(L).
\]

The last part of the statement now follows from the observation that $V_{\nu^+(L)+g}(K) \le V_{\nu^+(L)}(L) = 0$, hence $\nu^+(K) \le \nu^+(L)+g$ as desired.
\end{proof}

\begin{proof}[Proof of \cref{p:crossingchanges}]
The inequality $\nu^+(K_-) \le \nu^+(K_+)$ readily follows from \cite[Theorem 6.1]{BorodzikHedden}: the latter states that for each non-negative integer $n$ we have $V_n(K_-) \le V_n(K_+)$.
Applying the inequality with $n = \nu^+(K_+)$ we obtain $V_{\nu^+(K_+)}(K_-) \le 0$, hence $\nu^+(K_-) \le \nu^+(K_+)$, as desired.

The inequality $\nu^+(K_+) \le \nu^+(K_-)+1$ follows from \cref{l:gvsnu} above: in fact, there is a genus-$1$ cobordism from $K_-$ to $K_+$ obtain by smoothing the double point of the regular homotopy associated with the crossing change, and the previous lemma concludes the proof.
\end{proof}

We now turn to applications to other, more geometrically defined, concordance invariants, and we prove \cref{p:concordanceinvariants}.

\begin{proof}[Proof of \cref{p:concordanceinvariants}]
We need at least $\nu^+(K)$ negative crossing changes and at least $\nu^+(\overline K)$ positive crossing changes to turn $K$ into a knot $K_0$ such that $\nu^+(K_0) = \nu^+(\overline K_0) = 0$. In particular, we need to change at least $\nu^+(K) + \nu^+(\overline K)$ crossings to make $K$ slice, hence $u_s(K)\ge\nu^+(K) + \nu^+(\overline K)$.

As for the concordance unknotting number, one simply observes that $\nu^+(K)$ and $\nu^+(\overline K)$ are concordance invariants, hence every knot in the same concordance class of $K$ has unknotting number at least $\nu^+(K) + \nu^+(\overline K)$.

Finally, \cite[Proposition 2.1]{OwensStrle-unknotting} asserts that every immersed concordance can be factored into two concordances and a sequence of crossing changes. That is, given an immersed concordance from $K$ to the unknot with $c$ double points, there exist knots $K_0$ and $K_1$ such that $K_0$ is slice, $K_1$ is concordant to $K$, and there is a sequence of $c$ crossing changes from $K_0$ to $K_1$; from the proposition above, it follows that $c \ge \nu^+(K_0\#K_1) + \nu^+(\overline{K_0\#K_1}) = \nu^+(K)+\nu^+(\overline K)$.
\end{proof}

%% file: sect/subadd.tex

\section{Subadditivity of $\nu^+$}\label{s:subadd}

The goal of this section is proving \cref{t:subadd}. We start with a preliminary proposition.

\begin{prop}\label{p:subaddV}
For any two knots $K, L\subset S^3$ and any two non-negative integers $m,n$, we have
\[
V_{m+n}(K\#L) \le V_m(K) + V_n(L).
\]
\end{prop}

\begin{proof}
Consider the surgery diagrams in \cref{f:owensstrleX} and \cref{f:owensstrleW}, representing a closed 4--manifold $X$ and a 4--dimensional cobordism $W$ from $-S^3_{2(m+n)}(K\#L)$ to $-(S^3_{2m}(K)\# S^3_{2n}(L))$.
One should be careful with orientation-reversals here; in particular, notice that in \cref{f:owensstrleW} we represent the cobordism $\overline W$ obtained by turning $W$ upside down.

\begin{figure}[h]
\labellist
\pinlabel $2m$ at 19 10
\pinlabel $2n$ at 279 10
\pinlabel $K$ at 18 55
\pinlabel $L$ at 280 55
\pinlabel $0$ at 150 92
\pinlabel $0$ at 248 130
\pinlabel $\cup \;4$-handle at 350 130
\endlabellist
\centering
\includegraphics[scale=.7]{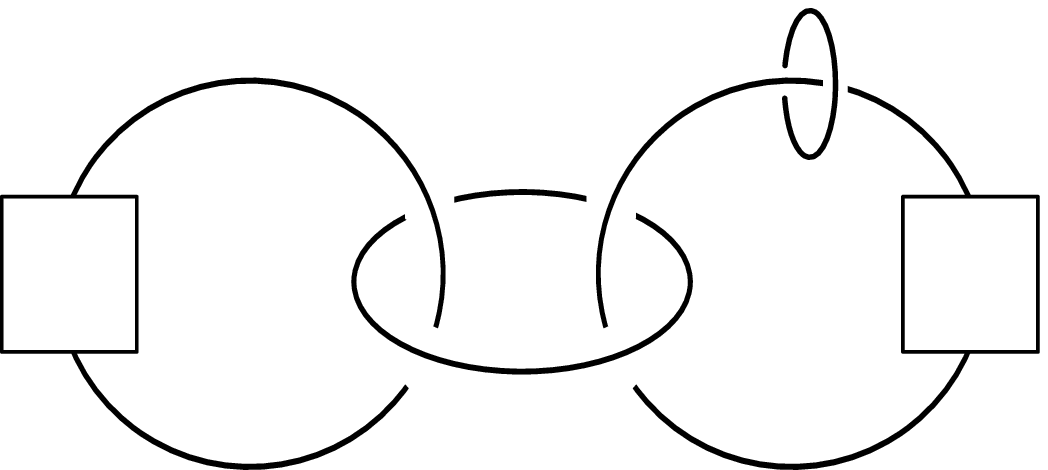}
\caption{The surgery diagram for the closed 4--manifold $X$.}\label{f:owensstrleX}
\end{figure}

\begin{figure}[h]
\labellist
\pinlabel $\langle 2m\rangle$ at 13 10
\pinlabel $\langle 2n\rangle$ at 285 10
\pinlabel $K$ at 18 55
\pinlabel $L$ at 280 55
\pinlabel $0$ at 150 92
\endlabellist
\centering
\includegraphics[scale=0.7]{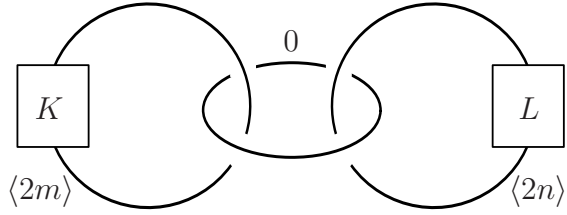}
\caption{The surgery diagram for the upside-down cobordism $\overline{W}$ from $S^3_{2m}(K)\# S^3_{2n}(L)$ to $S^3_{2(m+n)}(K\#L)$. The coefficients in brackets represent the negative boundary $\del_-W$.}\label{f:owensstrleW}
\end{figure}

As observed by Owens and Strle \cite{OwensStrle}, when $m,n>0$, $W\subset X$ is a negative definite cobordism from $S^3_{2m}(K)\# S^3_{2n}(L)$ to $S^3_{2(m+n)}(K\#L)$ with $H_2(W;\Z) = \Z$ and $\chi(W) = 1$.
When $m=0$ or $n=0$, $W$ has signature $\sigma(W)=0$; therefore, regardless of positivity of $m$ and $n$, $W$ is negative semidefinite.
Moreover, $W$ is obtained from $\del_-W$ by attaching a single $2$-handle, and this does not decrease the first Betti number of the boundary.
It follows that we are in the right setup to apply~\cite[Theorem 4.1]{Behrens}.

The 4--manifold $X$ is even; since 0 is a characteristic vector, it is the first Chern class of a \spinc structure $\fs_0$ on $X$.
The \spinc structure $\fs_0$ restricts to the \spinc structure on $W$ with trivial first Chern class, hence $c_1(\fs_0)^2=0$.

Notice also that $X\setminus W$ is the disjoint union of two 4--manifolds: one is the boundary connected sum of the surgery handlebodies for $S^3_{2m}(K)$ and $S^3_{2n}(L)$, and the other is the surgery handlebody for $S^3_{2(m+n)}(K\#L)$ with the reversed orientation.
In particular, labelling of the restriction of $\fs_0$ onto the two boundary components of $W$ is determined by the evaluation of $c_1(\fs_0)$ on the generators of the second cohomology of the two pieces \cite[Section 2.4]{OzsvathSzabo-integralsurgeries}.

With the chosen convention for labelling \spinc structures \eqref{e:spinclabel}, since $c_1(\fs_0) = 0$, $\fs_0$ restricts to the \spinc structure $\ft_m$ on $S^3_{2m}(K)$, to the \spinc structure $\ft_n$ on $S^3_{2n}(L)$, and to the \spinc structure $\ft_{m+n}$ on $S^3_{2(m+n)}(K)$.

We observe here that $\underline{d}(-S^3_{0}(\overline K),\ft_0)+\frac12b_1(-S^3_{0}(\overline{K})) = 2V_0(K)$, and the same holds for $L$ and $K\#L$ (compare with~\cite[Proposition 4.12]{OzsvathSzabo-absolutely} and~\cite[Proposition 3.8]{Behrens}).
When $m>0$, however, $\underline{d}(-S^3_{2m}(\overline K),\ft_m) = d(-S^3_{2m}(\overline K),\ft_m) = \frac14+2V_n(K)$, and analogous formulae hold for $L$ and $K\#L$.

We now apply \cite[Proposition 3.7 and Theorem 4.1]{Behrens} to $W$; we have:
\begin{gather*}
b_2^-(W) + 4d(-(S^3_{2m}(\overline K)\#S^3_{2n}(\overline L)),\ft_m\#\ft_n) + 2b_1(-(S^3_{2m}(\overline K)\#S^3_{2n}(\overline L))) \le \\ \le 4\underline{d}(-S^3_{2m}(\overline K),\ft_m)+2b_1(-S^3_{2m}(\overline K))+4\underline{d}(-S^3_{2n}(\overline L),\ft_n)+2b_1(-S^3_{2n}(\overline L)).
\end{gather*}

When $m$ and $n$ are both positive, the inequality above becomes:
\[
{\textstyle \frac14}+{\textstyle \frac14} + 2V_{m+n}(K\#L)  \le {\textstyle \frac14} + 2V_m(K) + {\textstyle \frac14} + 2V_n(L).
\]

When exactly one among $m$ and $n$ vanishes, say $m=0$, we have:
\[
{\textstyle \frac14} + 2V_{n}(K\#L)  \le  2V_0(K) + {\textstyle \frac14} + 2V_n(L),
\]

Finally, when $m=n=0$,
\[
2V_0(K\#L)  \le  2V_0(K) + 2V_0(L),
\]

In all cases, we have proved that $V_{m+n}(K\#L) \le V_m(K)+V_n(L)$, as desired.
\end{proof}

We are now ready to prove \cref{t:subadd}.

\begin{proof}[Proof of \cref{t:subadd}]
This now follows from \cref{p:subaddV} by setting $m=\nu^+(K)$ and $n=\nu^+(L)$. In fact, since $V_m(K) = V_n(L) = 0$,
\[
V_{m+n}(K\#L) \le V_m(K) + V_n(L) = 0;
\]
that is, $\nu^+(K\#L) \le m+n = \nu^+(K) + \nu^+(L)$.
\end{proof}

%% file: sect/examples.tex

\section{Examples}\label{s:examples}

In this section we study a 3--parameter family of pairs of torus knots on which the lower bound given by $\nu^+$ is sharp. We first start with a 1--parameter subfamily that we study in some detail, and we then turn to the whole family.

\begin{ex}
We are going to present an example in which the bound provided by $\nu^+$ on the genus of a cobordism between torus knots is stronger than the ones given by the Tristram--Levine signature function, $\tau$, $s$ and $\Upsilon$, and moreover it is sharp.

\begin{figure}[h]
\labellist
\pinlabel $\frac{2p}{3p}$ at 28 51
\pinlabel $\frac{2p}{3p}$ at 168 51
\pinlabel $=$ at 118 51
\pinlabel $\frac ap$ at 70 87
\pinlabel $\frac{a_1}p$ at 210 87
\pinlabel $\frac{a_2}p$ at 210 51
\pinlabel $\frac{a_3}p$ at 210 15
\endlabellist
\centering
\includegraphics[scale=1]{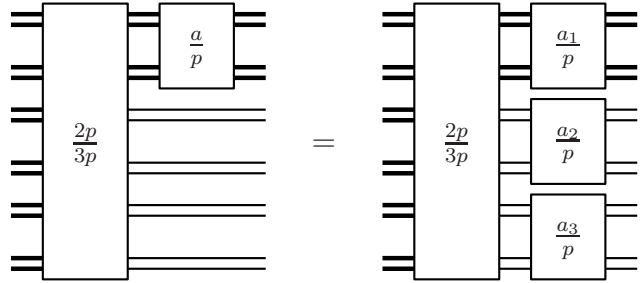}
\caption{The knot $K_{a,p}$: the boxes indicate the number of full twists. Equality holds whenever $a = a_1 + a_2 + a_3$}\label{f:cablingT32a}
\end{figure}

\begin{figure}[h]
\labellist
\pinlabel $\frac{3p}{2p}$ at 28 33
\pinlabel $\frac{3p}{2p}$ at 168 33
\pinlabel $=$ at 118 33
\pinlabel $\frac bp$ at 70 53
\pinlabel $\frac{b_1}p$ at 210 52
\pinlabel $\frac{b_2}p$ at 210 16
\endlabellist
\centering
\includegraphics[scale=1]{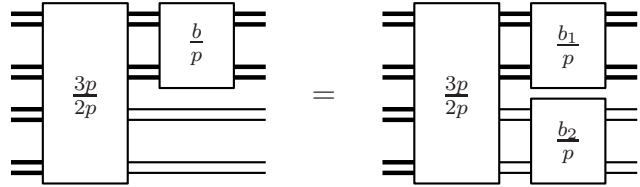}
\caption{The knot $K'_{b,p}$: the boxes indicate the number of full twists. Equality holds whenever $b = b_1 + b_2$}\label{f:cablingT32b}
\end{figure}

Define the two families of links $K_{a,p}$ and $K_{b,p}^\prime$ as the closure of the braids pictured in \cref{f:cablingT32a} and \cref{f:cablingT32b}.
Notice that $K_{a,p}$ and $K'_{b,p}$ are $(p,s)$--cables of the trefoil for some $s$, and that they are knots if and only if $\gcd(a,p)=1$ and $\gcd(b,p)=1$ respectively.
Moreover, $K_{a,p}$ is the product of $2p(3p-1) + a(p-1)$ positive generators of the braid group on $3p$ strands, hence its closure represents a transverse knot in the standard contact 3--sphere with self-linking number $6p^2 + (a-5)p-a$.
Since for closure of positive braids the self-linking number agrees with the Seifert genus, we can compute the cabling parameter $s = 6p+a$.
In conclusion, we have shown that $K_{a,p}$ is the $(p,6p+a)$--cable of $T_{2,3}$.

The same argument applies to $K'_{b,p}$, the self-linking number computation yields $6p^2 + (b-5)p-b$, hence showing that $K'_{b,p}$ is the $(p,6p+b)$--cable of $T_{2,3}$.
In particular $K_{a,p}$ and $K'_{b,p}$ are isotopic if and only if $a=b$.

Now consider the knots $K_{12,p} = K'_{12,p}$. Denote with $\sigma_i$ the $i$--th elementary generator of the braid group, and, whenever $i<j$, denote with $\sigma_{i,j}$ the product $\sigma_i\sigma_{i+1}\cdots\sigma_j$.

Setting $a_1 = a_2 = a_3 = 4$ in the right-hand side of \cref{f:cablingT32a} exhibits $K_{12,p}$ as the closure of the braid
\[
\sigma_{1,3p-1}^{2p} \cdot \overbrace{\underbrace{\sigma_{1,p-1}^4}_{\Sigma_1^4}\cdot \underbrace{\sigma_{p+1,2p-1}^4}_{\Sigma_2^4}\cdot  \underbrace{\sigma_{2p+1,3p-1}^4}_{\Sigma_3^4}}^\Pi.
\]
The three elements $\Sigma_1, \Sigma_2,\Sigma_3$ commute, hence $\Pi = (\Sigma_1\Sigma_2\Sigma_3)^4$. 
Now, notice that $\Sigma_1\sigma_p\Sigma_2\sigma_{2p}\Sigma_3 = \sigma_{1,3p-1}$. 
Since adding a generator $\sigma_i$ corresponds to a band attachment between two strands, we produce a cobordism built out of 8 bands from $K_{12,p}$ to $T_{2p+4,3p}$; if $p$ is coprime with 6, both ends of the cobordism are connected, and its genus is 4.

An analogous argument, setting $b_1 = b_2 = 6$ in the right-hand side of \cref{f:cablingT32b} produces a 6--band, genus-3 cobordism from $K'_{12,p}$ to $T_{2p,3p+6}$ whenever $p$ is coprime with 6.

Suppose now that $p\equiv 5 \pmod{6}$, $p\ge 11$. Gluing the two cobordisms above yields a genus-7 cobordism between $K=T_{2p+4,3p}$ and $L=T_{2p,3p+6}$.

Applying \cref{p:computenu} above we obtain a sharp bound on the slice genus; in fact, in the same notation as in \cref{p:computenu}, we have:
\begin{itemize}
\item $2\delta_K = 2g(K) = 6p^2+7p-3$ and $2\delta_L = 2g(L) = 6p^2+ 7p - 5$;
\item $\Gamma_K(2) = 3p$ and $\Gamma_L(2) = 3p+6$;
\item $\Gamma_K(3)=4p+8$ while $\Gamma_L(3)=4p$.
\end{itemize}

It follows that
\begin{align*}
\nu^+(K\#\overline L) &\ge 1 + \Gamma_L(2) - \Gamma_K(2) = 7;\\
\nu^+(L\#\overline K) &\ge -1 + \Gamma_K(3) - \Gamma_L(3) = 7.
\end{align*}
A direct computation using \cite[Theorem 1.3]{OzsvathSzaboStipsicz} shows that for $p=11,17,23,29$ the bound given by $\Upsilon$ is 3, the one given by the Tristram--Levine signatures is either 2 or 5, and the one given by $\tau$ and $s$ is 1.

Moreover, we need at least 7 positive and 7 negative crossing changes to turn $K$ into $L$, hence their Gordian distance is at least 14. Additionally, suppose that we have a factorisation of the cobordism above into genus-1 cobordisms, and suppose that one of these cobordisms goes from $K_1$ to $K_2$. 
Then both $\nu^+(K_2) = \nu^+(K_1)-1$ and $\nu^+(\overline K_2) = \nu^+(\overline K_1)-1$.
\end{ex}

\begin{ex}
We can promote the family above to a family parametrised by suitable triples of integers $(p,q,r)$ as follows: instead of considering the $(p,6p+12)$--cable of the trefoil $K_{12,p} = K'_{12,p}$ we can consider the $(p,qr(p+2))$--cable $K^p_{q,r}$ of $T := T_{q,r}$. The first condition we impose on the triple $(p,q,r)$ is that $q<r$ and $\gcd(q,r)=1$.

By looking at $K^p_{q,r}$ as a cable of $T$ seen as the closure of an $r$--braid, we can glue $2q\cdot (r-1)$ bands to $K^p_{q,r}$ and obtain $K = T_{q(p+2), rp}$. Call $x_1 = q(p+2), x_2 = rp$ the two generators of the semigroup $\Gamma_K$.

By looking at $K^p_{q,r}$ as a cable of $T$ seen as the closure a $q$--braid instead, we see that we can glue $2r\cdot(q-1)$ bands to $K^p_{q,r}$ and obtain $L = T_{qp, r(p+2)}$. Call $y_1 = qp, y_2 = r(p+2)$ the two generators of the semigroup $\Gamma_L$.

If $\gcd(p,2qr) = \gcd(p+2,2qr) = 1$, both $K$ and $L$ have one component, i.e. they are torus knots; e.g. both equalities hold if $p\equiv -1 \pmod{2qr}$.
Moreover, $\delta_K-\delta_L = g(K) - g(L) = r-q$, and above we produced a cobordism of genus $2qr-q-r$ between $K$ and $L$, made of $4qr-2q-2r$ bands.
Hence $\nu^+(K\#\overline L), \nu^+(L\#\overline K) \le 2qr-q-r$.

Choose $p$ sufficiently large; it is elementary to check that if $p\ge2qr-1$, for $n_1 = \delta_T+q-1$ and $n_2 = \delta_T+r-1$ we have
\begin{align*}
\Gamma_T(n_1) &= (q-1)r, &\Gamma_T(n_2) &= (r-1)q;\\
\Gamma_K(n_1) &= (q-1)x_2 = (q-1)rp,\quad &\Gamma_K(n_2) &= (r-1)x_1 = (r-1)q(p+2);\\
\Gamma_L(n_1) &= (q-1)y_2 = (q-1)r(p+2), &\Gamma_L(n_2) &= (r-1)y_1 = (r-1)qp.
\end{align*}

If we set $n=n_1$ in \cref{t:maincomputation} we obtain:
\[
\nu^+(K\#\overline L) \ge \delta_K - \delta_L + \Gamma_L(n_1)-\Gamma_K(n_1) =  2qr-q-r.
\]
Reversing the roles of $K$ and $L$ and setting $n=n_2$ yields
\[
\nu^+(L\#\overline K) \ge \delta_L - \delta_K + \Gamma_K(n_2)-\Gamma_L(n_2) = 2qr-q-r.
\]

The lower bound for the genus given by $\nu^+$ is in this case is tight, as the upper and lower bounds
match, and moreover the Gordian distance between $K$ and $L$ is at least $4qr-2q-2r$.

\end{ex}